\documentclass{amsart}
\usepackage{amsmath, amsfonts, amssymb, enumerate, ogonek} %

\def\k{\kappa}

\newcommand{\mN}{\mathcal{N}}

\newcommand{\mR}{\mathcal{R}}
\newcommand{\mS}{\mathcal{S}}

\newcommand{\fm}{\mathfrak{m}}

\newcommand{\fn}{\mathfrak{n}}

\newcommand{\bfC}{\mathbf{C}}

\newcommand{\bfF}{\mathbf{F}}

\newcommand{\bfQ}{\mathbf{Q}}

\newcommand{\bfT}{\mathbf{T}}

\newcommand{\bfZ}{\mathbf{Z}}

\newcommand{\Oo}{\mathcal{O}}

\newcommand{\ov}{\overline}

\newcommand{\be}{\begin{equation}}
\newcommand{\ee}{\end{equation}}
\newcommand{\bes}{\begin{equation*}}
\newcommand{\ees}{\end{equation*}}

\newcommand{\bs}{\begin{split}}
\newcommand{\es}{\end{split}}
\newcommand{\bss}{\begin{split*}}
\newcommand{\ess}{\end{split*}}

\newcommand{\bmat}{\left[ \begin{matrix}}
\newcommand{\emat}{\end{matrix} \right]}
\newcommand{\bsmat}{\left[ \begin{smallmatrix}}
\newcommand{\esmat}{\end{smallmatrix} \right]}

\newcommand{\bml}{\begin{multline}}
\newcommand{\eml}{\end{multline}}
\newcommand{\bmls}{\begin{multline*}}
\newcommand{\emls}{\end{multline*}}

\DeclareMathOperator{\Ann}{Ann}

\DeclareMathOperator{\diag}{diag}

\DeclareMathOperator{\End}{End}

\DeclareMathOperator{\SL}{SL}
\DeclareMathOperator{\Sp}{Sp}

\DeclareMathOperator{\U}{U}
\DeclareMathOperator{\val}{val}

\newcommand{\hf}{\hspace{5pt}}

\usepackage[all]{xy}
\usepackage{amsthm}
\usepackage{amssymb, amsfonts}
\SelectTips{cm}{10}\UseTips
\theoremstyle{plain}
\newtheorem{thm}{Theorem}
\newtheorem{prop}[thm]{Proposition}

\newtheorem{lemma}[thm]{Lemma}

\theoremstyle{definition}

\newtheorem{example}[thm]{Example}
\newtheorem{rem}[thm]{Remark}

\numberwithin{thm}{section}
\numberwithin{equation}{section}

\begin{document}
\title[On higher congruences between automorphic forms]{On higher congruences between automorphic forms}
\author{Tobias Berger$^1$ \and
Krzysztof Klosin$^2$ \and Kenneth Kramer$^2$}
\address{$^1$School of Mathematics and Statistics, University of Sheffield, Hicks Building, Hounsfield Road, Sheffield S3 7RH, UK.}
\address{$^2$Department of Mathematics,
Queens College,
City University of New York,
65-30 Kissena Blvd,
Queens, NY 11367, USA.}



\date{10 February 2013}

\begin{abstract}

We prove a commutative algebra result which has consequences for congruences between automorphic forms modulo prime powers. If $C$ denotes the congruence module for a fixed automorphic Hecke eigenform $\pi_0$ we prove an exact relation between the $p$-adic valuation of the order of $C$ and the sum of the exponents of $p$-power congruences between the Hecke eigenvalues of $\pi_0$ and other automorphic forms. We apply this result to several situations including the congruences described by Mazur's Eisenstein ideal. 
\end{abstract}


\maketitle

\section{Introduction}
Let $p$ be a prime.
 Let $f_1, \dots, f_r$ be all weight $2$  normalized simultaneous eigenforms of level $\Gamma_0(N)$ with $N$ prime. A famous result of Mazur's (\cite{Mazur78} Proposition 5.12(iii)) states that at least one of these forms  is congruent modulo $p$ to the Eisenstein series $E_2^*=1-N+24 \sum_{n=1}^{\infty} \sigma^*(n) q^n$ for $$\sigma^*(n)=\sum_{0<d\mid n, (d,N)=1} d$$ if $p$ divides the numerator $\mN$ of $\frac{N-1}{12}$. One may ask for the precise relation between $\val_p(\mN)$ and the `depth' of congruence of the newforms $f_1, \dots, f_r$ to $E_2^*$. One of our results (Proposition \ref{Version 1}) provides an answer to this question. More precisely, if $\varpi_N$ is a uniformizer in the valuation ring of a finite extension of $\bfQ_p$ (of ramification index $e_N$) which contains all of the eigenvalues of the $f_i$s and $m_i$ is defined as the largest integer such that $E_2^*\equiv f_i$ mod $\varpi_N^{m_i}$, then \be \label{intro1}\frac{1}{e_N}\sum_{i=1}^r m_i = \val_p(\mN).\ee 

In general the 
 Hecke eigenvalue congruences between a fixed automorphic eigenform $\pi_0$ on a reductive algebraic group $G$ and other eigenforms (call the set of them $\Pi$) on $G$ 
are controlled by the order of the so called \emph{congruence module}. This module can be defined as the quotient of the Hecke algebra $\bfT$ acting on the forms in $\Pi$ by the image $J$ in $\bfT$ of the annihilator of $\pi_0$ in the Hecke algebra $\bfT_0$ acting on the forms in $\Pi \cup \{\pi_0\}$
 (for details on this setup see section \ref{Congruences - the general case}). In the more classical situation, where $\pi_0$ is an Eisenstein series, $J$ is often called the \emph{Eisenstein ideal}. However, our considerations apply in a much more general framework, e.g., when $J$ corresponds to a primitive cusp form or to a CAP automorphic representation (see section \ref{s5}).

 In this general setup one may ask again how many of the automorphic forms in $\Pi$ have Hecke eigenvalues congruent to those of $\pi_0$ and what the ``depth''
of these congruences is. We prove that whenever the ideal $J$ is principal (like in the Mazur example above) then an equality analogous to (\ref{intro1}) can be achieved with $\#\bfT/J$ replacing $\mN$. However, if the ideal $J$ is not principal, we show that instead of equality the analog $S$ of the quantity on the left of (\ref{intro1}) 
(the total ``depth" of the congruences) is bounded from below by $\val_p(\# \bfT/J)$ (cf. Proposition \ref{mainprop}). Let us note here that in general principality of $J$ is not known or even expected to hold. In many cases it is conjectured or known that the order $\# \bfT/J$ is related to a special value of an appropriate $L$-function. In this case, our results provide an $L$-value bound on $S$. 

Our paper is organized as follows. In section \ref{Commutative algebra} we prove a commutative algebra result which will be the basis for all of our applications to congruences among automorphic forms. 
In fact it can be seen as a more general result regarding the distribution of congruences among various subsets of automorphic forms (not just single automorphic forms) - for an application in this context to a modularity problem see \cite{BergerKlosin13preprint}. In section \ref{s3} we apply this result to Mazur's congruences and derive (\ref{intro1}). In section \ref{Congruences - the general case} we set up a framework to deal with general type of congruences among automorphic forms and prove a lower bound on the 
total depth of congruences in terms of the order of a congruence module. Finally, in section \ref{s5} we provide more examples. 

On completion of this paper we learned from Fred Diamond that a similar argument was already used by him and Wiles to prove an analogue of the inequality in our Proposition \ref{mainprop}  in the case of congruences between cusp forms of different levels (cf. Theorem 4C in \cite{Diamond89} and Lemma 1.4.3 in \cite{Wiles88}). 


We would like to express special thanks to Gabor Wiese for his comments on the paper and for help with verifying some numerical examples. We would also like to thank Jim Brown for an email correspondence regarding multiplicity one for Saito-Kurokawa lifts.

\section{Commutative algebra} \label{Commutative algebra}

Let $p$ be a prime. Let $\Oo$ be the valuation ring of a finite extension $E$ of $\bfQ_p$. Fix a choice of a uniformizer $\varpi$ of $\Oo$ and write $\bfF=\Oo/\varpi \Oo$ for the residue field. 

Let $s \in \bfZ_+$ and let $\{n_1, n_2, \dots, n_s\}$ be a set of $s$ positive integers. Set $n=\sum_{i=1}^s n_i$.
 Let $A_i=\Oo^{n_i}$ with $i\in \{1,2,\dots, s\}$.  Set $A=\prod_{i=1}^s A_i=\Oo^n$. Let $\varphi_i: A\twoheadrightarrow A_i$ be the canonical projection. Let $T \subset A$ be a (local complete) $\Oo$-subalgebra which is of full rank as an $\Oo$-submodule and let $J \subset T$ be an ideal of finite index. Set $T_i=\varphi_i(T)$ and $J_i=\varphi_i(J)$. Note that each $T_i$ is also a (local complete) $\Oo$-subalgebra and the projections $\varphi_i|_{T}$ are local homomorphisms. Then $J_i$ is also an ideal of finite index in $T_i$.
\begin{thm} \label{Kr12} If $\# \bfF^{\times} \geq s-1$ and each $J_i$ is principal, then $\# \prod_{i=1}^s T_i/J_i \geq \# T/J$.\end{thm}
\begin{rem} \label{strict} Note that the inequality in Theorem \ref{Kr12} may be strict. Consider for example $T=\{(a,b)\in \Oo \times \Oo \mid a \equiv b \pmod{\varpi}\} \subset \Oo \times \Oo=A$ with $A_i=\Oo$ ($i=1,2$). Let $J=\{(\varpi a, \varpi b) \in \Oo \times \Oo\mid a,b \in \Oo\}$ be the maximal ideal of $T$. Then $T/J\cong T_1/J_1\cong T_2/J_2 \cong \Oo/\varpi$. Let us also note that Theorem \ref{Kr12} is false if $J$ is only assumed to be an $\Oo$-submodule of $T$, but not an ideal. Indeed, let $T$ be as above and $J=\{(a,b)\in \Oo \times \Oo \mid a \equiv b \pmod{\varpi^2}\}$. Then $T/J \cong \Oo/\varpi$, but $T_1/J_1\cong T_2/J_2\cong 0$.
 \end{rem} 
As explained in Remark \ref{strict} the inequality in Theorem \ref{Kr12} may be strict, however this is not the case when $J$ itself is principal. For principal $J$ one not only obtains equality, but also the assumption on the residue field is unnecessary. Let us state this as a separate proposition. 
\begin{prop} \label{principalJ} If $J$ is principal, then $\# \prod_{i=1}^s T_i/J_i = \# T/J$.\end{prop}
We prepare the proofs of Theorem \ref{Kr12} and Proposition \ref{principalJ} by two lemmas and the proposition will follow from these alone, while the proof of the Theorem is more difficult. 


 If $M$ is any finitely presented  $\Oo$-module, we will write ${\rm Fit}_{\Oo}(M)$ for its Fitting ideal (see for example \cite{MazurWiles84}, Appendix for a treatment of Fitting ideals). 
\begin{lemma} \label{fit} Let $M$ be an $\Oo$-module of finite cardinality.  Then $\# M  = \# \Oo/{\rm Fit}_{\Oo}(M)$. \end{lemma}
\begin{proof} This is easy. \end{proof}
\begin{lemma} \label{principallemma} Let $\alpha = (\alpha_1, \alpha_2, \dots, \alpha_s) \in J$ be such that $\alpha_i$ generates $J_i$ as an ideal of $T_i$ for each $1 \leq i \leq s$. Then $\#T/J \leq \# T/\alpha T=\# \prod_{i=1}^s T_i/\alpha_iT_i$. \end{lemma}
\begin{proof} The first inequality is clear. 
Let $B \in M_n(\Oo)$ be such that $T=B\Oo^n$. Write $\alpha_i = (\beta_1^i, \beta_2^i, \dots, \beta_{n_i}^i)\in A_i$. Let $$A=\diag(\beta^1_1, \beta^1_2, \dots, \beta^1_{n_1}, \beta^2_1, \beta^2_2, \dots, \beta^2_{n_2}, \dots, \beta^s_1, \beta^s_2, \dots, \beta^s_{n_s})\in\Oo^n.$$ Note that $AT=\alpha T=AB\Oo^n$. Thus, we have ${\rm Fit}_{\Oo}(\Oo^n/T) = (\det B) \Oo$ and ${\rm Fit}_{\Oo}(\Oo^n/\alpha T) = (\det AB)\Oo$. It follows that $$\# T/\alpha T = \#\Oo/\det A = \# \Oo/\left(\prod_{i=1}^s\prod_{j=1}^{n_i} \beta_j^i\right).$$

Now let us compute $\#T_1/\alpha_1 T_1$. For a matrix $M \in M_n(\Oo)$ write $M_1$ for the first $n_1$ rows of $M$. Note that $T_1=B_1\Oo^n$ and $\alpha_1 T_1 = A_1 T_1 = A_1B \Oo^n$. We have ${\rm Fit}_{\Oo}(\Oo^{n_1}/T_1) = {\rm Fit}_{\Oo}(\Oo^{n_1}/B_1 \Oo^n)$ and ${\rm Fit}_{\Oo}(\Oo^{n_1}/J_1) = {\rm Fit}_{\Oo}(\Oo^{n_1}/A_1B \Oo^n)$. Note that every entry in any row (say $j$th row) of $A_1B$ equals the corresponding entry in the $j$th row of $B_1$ times $\beta^1_j$. Thus the determinant of every $n_1 \times n_1$ minor of $A_1B$ equals the determinant of the corresponding minor of $B_1$ times $\prod_{j=1}^{n_1} \beta^1_j$. Hence by the definition of the Fitting ideals we get $${\rm Fit}_{\Oo}(\Oo^{n_1}/J_1)=\left(\prod_{j=1}^{n_1} \beta^1_j\right)  {\rm Fit}_{\Oo}(\Oo^{n_1}/T_1).$$ Thus we conclude that $\#T_1/\alpha_1 T_1 =\#\Oo/\prod_{j=1}^{n_1} \beta^1_j$.  Analogous argument works for all $1<i\leq s$. \end{proof}
\begin{proof} [Proof of Proposition \ref{principalJ}] Take $\alpha$ in Lemma \ref{principallemma} to be a generator of $J$. \end{proof}

\begin{proof}[Proof of Theorem \ref{Kr12}]
Assume that each $J_i$ is principal. Let $\alpha_i$ be any generator of $J_i$. Note that $\alpha_1$ is not a zero divisor in $T_1$ since it is of the form $(a_1, \dots, a_{n_1})$ with $a_i\in \Oo$. If $\alpha_1$ was to be a zero-divisor one of the $a_i$ must be zero, but then $J_1$ is not of finite index. Also $\alpha_i$ is not a zero-divisor in $T_i$ for $i>1$. 
\begin{prop} \label{Kr11} Assume $\# \bfF^{\times} \geq s-1$. There exists $\alpha_i \in J_i$ such that $(\alpha_1, \dots, \alpha_s) \in J$ and each $\alpha_i$ generates $J_i$. \end{prop}

\begin{proof} We proceed by induction on $s$. The case $s=1$ is clear. Assume the statement is true for $s=n\geq 1$. Let $s=n+1$. Let $T'$ (resp. $J'$) be the image of $T$ (resp. $J$) under the projection $A \twoheadrightarrow \prod_{i=1}^n T_i$. Then by the inductive hypothesis we know there exists $(\alpha_1, \dots, \alpha_n) \in J'$ such that each $\alpha_i$ generates $J_i$. Since $J'$ is exactly the image of $J$ we know that there exists $x \in T_{n+1}$ such that  $(\alpha_1, \dots, \alpha_n, x) \in J$. By symmetry the inductive hypothesis also gives an element $(x', \beta_2, \dots, \beta_{n+1})\in J$ such that each $\beta_i$ generates $J_i$.

There exists $z_{n+1} \in T_{n+1}$ such  that $x=z_{n+1} \beta_{n+1}$. Lift $z_{n+1}$ to $(z_1, \dots, z_{n+1}) \in T$ and consider the difference:
\begin{multline*}(\alpha_1, \alpha_2,\dots, x) - (z_1, z_2,\dots, z_{n+1})\cdot (x', \beta_2, \dots, \beta_{n+1}) = \\
=(\alpha_1 - z_1 x', \alpha_2-z_2\beta_2, \dots, \alpha_n-z_n \beta_n,0).\end{multline*} There exists $u_1 \in T_1$ and (if $n>1$) $u_i, u'_i \in T_i$ ($i=2,3,\dots, n$) such that $$\alpha_1- z_1 x' = u_1 \alpha_1 \quad \textup{and} \quad \alpha_i-z_i\beta_i = u_i \alpha_i = u'_i\beta_i,\hf i=2,3,\dots, n.$$
This gives us $$z_1x'= (1-u_1) \alpha_1  \quad \textup{and} \quad z_i \beta_i = (1-u_i) \alpha_i\quad \textup{and} \quad \alpha_i=(u'_i+z_i)\beta_i, \hf i=2,3,\dots, n.$$

First let us assume that $n>1$ and consider the last set of equations first. It implies that $u'_i+z_i \in T_i^{\times}$ ($i=2,3,\dots, n$). This means either $u'_i$ or $z_i$ is a unit. If for any $i=2,3,\dots,n$ one has $z_i\in T_i^{\times}$, then $(z_1, z_2,\dots, z_{n+1}) \in T^{\times}$,  because a non unit cannot map to a unit under a local homomorphism. But then we get that $z_{n+1} \in T_{n+1}^{\times}$ and hence $x$ generates $J_{n+1}$, so $(\alpha_1, \alpha_2, \dots, \alpha_n,x)\in J$ has the property that all the coordinates generate the corresponding $J_i$ and this concludes the inductive step in this case.

If $z_i \not\in T_i^{\times}$ for any $i=2,3,\dots, n$, then $u'_i \in T_i^{\times}$ for all $i=2,3,\dots, n$. Then $\alpha_i-z_i \beta_i$ generates $T_i$ for all $i=2,3\dots, n$. 

Now, let us come back to the general case $n \geq 1$. 
First assume that  $u_1\in \fm_{T_1}$. Since $T_1$ is local and complete, this implies that $1-u_1 \in T_1^{\times}$ (cf. \cite{Eisenbud}, Proposition 7.10). Moreover, $x' \in J_1$, so there exists $z \in T_1$ such that $x'=z\alpha_1$. So, we have $$z_1 z \alpha_1 = z_1 x' = (1-u_1) \alpha_1.$$ Since $\alpha_1$ is a generator of $J_1$, we know it is a non-zero divisor, so, we can cancel it and get $$z_1z = 1-u_1 \in T_1^{\times}.$$ Thus both $z_1 \in T_1^{\times}$ and $z \in T_1^{\times}$. This implies that $J_1 = \alpha_1 T_1 = x' T_1$ (i.e., $x'$ also generates $J_1$). This implies that $(x', \beta_2, \dots, \beta_{n+1})\in J$ has the property that each coordinate generates the respective ideal (this  concludes the inductive step in this case). 
Now assume $u_1 \in T_1^{\times}$. Then $\alpha_1 - z_1 x'$ also generates $J_1$. 

This shows that either we have an element that satisfies the hypothesis of Proposition \ref{Kr11} or we have one consisting of generators of the ideals $J_i$ for $i\leq n$ on the first $n$ coordinates and zero on the last one. By symmetry we have proved the following lemma. 
\begin{lemma} \label{L1} Either there exists $(\alpha_1, \dots, \alpha_{n+1}) \in J$ such that $\alpha_iT_i=J_i$ for all $i$ or there exist elements $$a_1:=(0,\alpha^1_2,\dots, \alpha^1_{n+1}),a_2:= (\alpha^2_1, 0, \dots, \alpha^2_{n+1}),\dots, a_{n+1}:=(\alpha^{n+1}_1, \alpha^{n+1}_2, \dots, 0)\in J$$ such that $\alpha_i^jT_i=J_i$ for all $i\neq j$. \end{lemma}

Assume that there is no element $(\alpha_1, \dots, \alpha_{n+1}) \in J$ such that $\alpha_iT_i=J_i$ for all $i$. Then by Lemma \ref{L1} we get the elements described in the second part of the lemma. If $n=1$, then $a_1 + a_2\in J$ and each of its coordinates generates the corresponding $J_i$. This completes the inductive step in this case. 

Now assume $n>1$.  
 For every $2 \leq i \leq n$ consider the set $\Sigma_i=\{\alpha^1_i + u \alpha^{n+1}_i\}$, where $u\in \Oo^{\times}$ runs over the set of representatives $S$ of $\bfF^{\times}$. We claim that there exist at least $\#\bfF^{\times}-1$ elements of $\Sigma_i$ such that each of them generates $J_i$. Indeed, suppose that there exist $a,b \in S$ such that $\alpha^1_i + a \alpha^{n+1}_i=\nu_1$ and $\alpha^1_i + b \alpha^{n+1}_i=\nu_2$ and both $\nu_1$ and $\nu_2$ do not generate $J_i$. Then since $\nu_1-\nu_2$ also does not generate $J_i$, we get that $(a-b)\alpha^{n+1}_i$ does not generate $J_i$. However since $\alpha^{n+1}_i$ generates $J_i$, we must have that $\varpi\mid (a-b)$, hence $a=b$. 

We conclude that there are at least $\#\bfF^{\times}-(n-1)$ elements $u \in S$  such that the element $a_1 + u a_{n+1}\in J$ and each of its $n+1$ coordinates generates the corresponding $J_i$. This contradicts the assumption that no such element exists and concludes the inductive step in this last case. 
 \end{proof}
We are now in a position to complete the proof of Theorem \ref{Kr12}. Indeed, let $\alpha:=(\alpha_1, \dots, \alpha_s) \in J$ be as in Proposition \ref{Kr11}. Then $\alpha$ satisfies the assumptions of Lemma \ref{principallemma} and we conclude that $\# T/J \leq\#\prod_{i=1}^s T_i/J_i $ as claimed. \end{proof}

\section{Eisenstein congruences among elliptic modular forms} \label{s3}
We note the following application to Eisenstein congruences for elliptic modular forms:
Let $f_1, \ldots f_r$ be all weight 2  normalized simultaneous eigenforms for $\Gamma_0(N)$ for $N$ prime. Mazur proved in  \cite{Mazur78} Proposition 5.12(iii)) that if $p$ divides the numerator of $\frac{N-1}{12}$ then at least one of these forms  is congruent modulo $p$ to the Eisenstein series $E_2^*=1-N+24 \sum_{n+1}^{\infty} \sigma^*(n) q^n$ for $$\sigma^*(n)=\sum_{0<d\mid n, (d,N)=1} d.$$ 

From now on for the rest of the article fix an embedding $\ov{\bfQ}_p \hookrightarrow \bfC$.
For $1\leq i \leq r$ write $K_{f_i}$ for the (finite) extension of $\bfQ_p$ generated by the Hecke eigenvalues of $f_i$. 
Let $\mathcal{O}_N$ to be the ring of integers in the composition of all the coefficient fields $K_{f_i}$ and write $\varpi_N$ for a choice of uniformizer, $e_N$ for the ramification index of $\Oo_N$ over $\bfZ_p$ and $d_N$ for the degree of its residue field over $\bfF_p$. We have the following result regarding the exponents of the Eisenstein congruences modulo powers of $\varpi_N$:

\begin{prop}\label{Version 1}
For $i=1, \ldots, r$ let $\varpi_N^{m_i}$ be the highest power of $\varpi_N$ such that the Hecke eigenvalues of $f_i$ are congruent to those of $E_2^*$ modulo $\varpi_N^{m_i}$ for Hecke operators $T_{\ell}$ for all primes $\ell\nmid N$.
  Then $\frac{1}{e_N}(m_1+ \ldots +m_r)$
 is equal to the $p$-valuation of the numerator of $\frac{N-1}{12}$.
\end{prop}

\begin{proof} Denote by $S_2(N)$ the $\bfC$-space of modular forms of weight 2 and level $\Gamma_0(N)$. For any subring $R\subset \bfC$ write $\bfT_R$ for the $R$-subalgebra of $\End_{\bfC}(S_2(N))$ generated by the Hecke operators $T_{\ell}$ for all primes $\ell \nmid N$. Let $J_R$ be the Eisenstein ideal, i.e., the ideal of $\bfT_R$ generated by  $T_{\ell}-(1+\ell)$ for $\ell \nmid N$. For a prime ideal $\fn$ of $\bfT_R$ write $\bfT_{R,\fn}$ for the localization of $\bfT_R$ at $\fn$ and set $J_{R, \fn}:= J_R \bfT_{R,\fn}$. 

Note that it follows from the definition of $J_R$ that the $R$-algebra structure map $R\to \bfT_R/J_R$ is surjective. Hence if $R$ is a local ring with maximal ideal $\fm_R$, then the ideal $J_R + \fm_R \bfT_R$ is the unique maximal ideal of $\bfT_R$ containing $J_R$. 
To ease notation in the proof write $\Oo$ for $\Oo_N$ and $\varpi$ for $\varpi_N$. Let $\fm$ be the unique maximal ideal of $\bfT_{\Oo}$ containing $J_{\Oo}$. 
Renumber $f_i$s if necessary so that $m_i>0$ for $0 \leq i \leq s\leq r$ and $m_i=0$ for $s < i \leq r$. Note that one has $\bfT_{\Oo}/J_{\Oo} \cong \bfT_{\Oo, \fm}/J_{\Oo, \fm}$. 

By Theorem II.18.10 in \cite{Mazur78}, the ideal $J$ is principal, hence we
apply Proposition \ref{principalJ} with $T=\bfT_{\Oo, \fm}$ and
let $T_i=\Oo$ 
(i.e., we take $n_1=n_2=\dots=n_s=1$ with $n_i$s as in section \ref{Commutative algebra}) and $\varphi_i:T \to T_i$ the map sending a Hecke operator to its eigenvalue corresponding to $f_i$. Set  $\fm_{\bfZ_p} = \fm \cap \bfZ_p$. One has $\bfT_{\Oo, \fm_{\Oo}} = \bfT_{\bfZ_p, \fm_{\bfZ_p}}\otimes_{\bfZ_p} \Oo$ (\cite{DDT}, Lemma 3.27 and Proposition 4.7) and $J_{\Oo, \fm} = J_{\bfZ_p, \fm_{\bfZ_p}}\otimes_{\bfZ_p} \Oo$. Since $\Oo/\bfZ_p$ is a flat extension, one has $$\bfT_{\Oo, \fm}/J_{\Oo, \fm} \cong \frac{ \bfT_{\bfZ_p, \fm_{\bfZ_p}}\otimes_{\bfZ_p} \Oo}{J_{\bfZ_p, \fm_{\bfZ_p}}\otimes_{\bfZ_p} \Oo} \cong ( \bfT_{\bfZ_p, \fm_{\bfZ_p}} / J_{\bfZ_p, \fm_{\bfZ_p}}) \otimes_{\bfZ_p} \Oo.$$ 
Thus $$\val_p(\# T_i/J_i) = 
\val_p\left(\#\Oo/\varpi^{m_i}\Oo \right)
= m_i d_N = m_i \frac{[\Oo:\bfZ_p]}{e_N}.$$
On the other hand $$\val_p(\# T/J) = \val_p\left(\# \frac{\bfT_{\bfZ_p, \fm_{\bfZ_p}}}{J_{\bfZ_p, \fm_{\bfZ_p}}}\otimes_{\bfZ_p}\Oo\right)=[\Oo: \bfZ_p] \val_p\left(\# \frac{\bfT_{\bfZ_p, \fm_{\bfZ_p}}}{J_{\bfZ_p, \fm_{\bfZ_p}}}\right)$$
 By \cite{Mazur78} Proposition II.9.6 we know that $\val_p\left(\# \frac{\bfT_{\bfZ_p, \fm_{\bfZ_p}}}{J_{\bfZ_p, \fm_{\bfZ_p}}}\right)$
equals the $p$-adic valuation of the numerator of $\frac{N-1}{12}$.  
[Indeed, note that while Mazur's Hecke algebra includes the operator $T_N$ and ours does not, it makes no difference since $T_N$ acts as identity. This was observed by Calegari and Emerton - see Proposition 3.19 of \cite{CalegariEmerton05} - and in fact  follows from Proposition II.17.10 of \cite{Mazur78}.] This implies that the $p$-valuation of the order of $T/J$ is given by the $p$-valuation of the numerator of $\frac{N-1}{12}$ times $[\Oo:\bfZ_p]$. By Proposition \ref{principalJ} we have  $\val_p(\# T/J) =\sum_{i=1}^s \val_p(T_i/J_i)$, hence the Proposition follows. \end{proof}


\begin{rem}
\begin{enumerate}\item  A very similar statement to that of the proposition was posed as Question 4.1 in 
\cite{Wiese10}. One of the consequences of Proposition \ref{Version 1} is an affirmative answer to that question.
\item Recently, such higher Eisenstein congruences were also investigated numerically by Naskr\k{e}cki \cite{Naskrecki12preprint} (but demanding congruence of all the Hecke eigenvalues). He proves an upper bound for the exponent of the congruence for particular cuspforms and conjectures a stronger one in the case that the coefficient field is ramified. We note that as opposed to Naskr\k{e}cki our set $f_1, \dots,
f_r$ includes all cuspidal eigenforms and not just representatives of
distinct Galois conjugacy classes. 
\end{enumerate}
\end{rem}

\begin{example}
A particular example that illustrates the statement of Proposition \ref{Version 1} can be taken from Section 19 of \cite{Mazur78}: Let $p=2$ and $N=113$. Mazur states that in this case $\val_p\left(\# \frac{\bfT_{\bfZ_p, \fm_{\bfZ_p}}}{J_{\bfZ_p, \fm_{\bfZ_p}}}\right)$=2 and ${\rm rank}_{\bfZ_p} (\bfT_{\bfZ_p, \fm_{\bfZ_p}})=3$.

One checks that there is one cuspform defined over $\bfZ_p$ congruent to the Eisenstein series modulo $p$ and that over a ramified quadratic extension of $\bfZ_p$ there are a further two Galois conjugate cuspforms congruent to the Eisenstein series modulo a uniformizer of that extension. Using the notation of Proposition \ref{Version 1} we therefore have $e_{113}=2, m_1=2, m_2=m_3=1$, and indeed $\frac{1}{2} (2+1+1)=2=\val_p \frac{113-1}{12}$.



\end{example}

\begin{example}
Other examples can be taken from \cite{Naskrecki12preprint}. In particular in Table 5 of [loc.cit.] Naskr\k{e}cki considers the following example: Let $p=5$ and $N=31$. In this case there are two Galois conjugate cuspforms over a ramified extension of degree 2, and $m_1=m_2=1$, so that $\frac{1}{2}(1+1)=1=\val_p \frac{31-1}{12}$. He also considers an example with $p=5$ and $N=401$. In this case there is only one Galois conjugacy class of newforms (containing two forms $f_1, f_2$) congruent to the Eisenstein series whose Fourier coefficients generate an order in the ring of integers of $K_{f_1}K_{f_2}$. In this case  $e_{401}=1$, and $m_1=m_2=1$. So one has $1+1=2=\val_p \frac{401-1}{12}$.
\end{example}

\section{Congruences - the general case} \label{Congruences - the general case}

The results of section \ref{Commutative algebra} can be applied in the context of congruences among automorphic forms on a general reductive group whenever there is  a $p$-integral structure on the Hecke algebra.
It is also worth noting that in many situations one does not know or even does not expect the ideal $J$ to be principal. In this section we introduce a general framework of working with such congruences and prove an analogue of Proposition \ref{Version 1} (Proposition \ref{mainprop}) which uses Theorem \ref{Kr12} instead of Proposition \ref{principalJ}. 

Let $p$ be a prime as before. 
Let $E$ be a finite extension of $\bfQ_p$. Write $\Oo$ for its valuation ring and $\varpi$ for a choice of a uniformizer. Let $e$ be the ramification index of $\Oo$ over $\bfZ_p$ and set $d:=[\Oo/\varpi \Oo: \bfF_p]$. 
Let $\mS_0$ be a (finite dimensional) $\bfC$-space whose elements are modular forms (or more generally automorphic forms for a reductive algebraic group). We assume that there is some naturally defined $\Oo$-lattice $\mS_{0, \Oo}$ inside $\mS_0$. Such a lattice may consist for example of forms with Fourier coefficients in $\Oo$ or of images of cohomology classes with coefficients in $\Oo$ under an Eichler-Shimura-type isomorphism. 

Let $\bfT_0 \subset \End_{\Oo}\mS_{0,\Oo}$ be a commutative $\Oo$-subalgebra of endomorphisms which can be simultaneously diagonalized over $E$. Write $\Pi_0$ for the set of systems of eigenvalues of $\bfT_0$, i.e., for the set of $\Oo$-algebra homomorphisms $\lambda: \bfT_0 \twoheadrightarrow \Oo$. Then we can identify $\bfT_0$ with an $\Oo$-subalgebra of $\Oo^{\# \Pi_0}$ of finite index in a natural way.


Let $\lambda_0$ be a fixed element of $\Pi_0$. Set $\Pi:= \Pi_0 \setminus \{\lambda_0\}$.  
Let $\bfT$ be the image of $\bfT_0$ under the projection $\varphi_0:\prod_{\lambda \in \Pi_0} \Oo \twoheadrightarrow \prod_{\lambda \in \Pi} \Oo$.
If necessary we extend $E$ further so that $\# (\Oo/\varpi \Oo)^{\times} \geq \# \Pi - 1$. 
Let $\pi_0 \in \mS_{0, \Oo}$ be an eigenform corresponding to $\lambda_0$, i.e., such that $T\pi_0 = \lambda_0(T) \pi_0$ for every $T \in \bfT_0$.
Let $J \subset \bfT$ be the ideal $\varphi_0(\Ann_{\bfT_0}\pi_0)$.

\begin{rem} \label{multone} Note that the quotient $\bfT/J$ is finite. Indeed,
for every $\lambda \neq \lambda_0$ there is $T_{\lambda} \in J$ such that $\lambda(T_{\lambda}) \neq 0$. Since $\lambda(\bfT)\cong\Oo$, we conclude that $\#\lambda(\bfT)/\lambda(J)<\infty$. Note that it follows from the definition of $J$ that the structure map $\iota: \Oo \to \bfT/J$ is surjective. Hence if $\bfT/J$ is infinite, $\iota$ must be an isomorphism. However, the (finitely many) maps $\lambda\in \Pi$ account for all $\Oo$-algebra maps from $\bfT$ to $\Oo$.
 Hence there must exist $\lambda\neq \lambda_0$ such that $\lambda$ factors through $\bfT/J$ which implies that $\lambda(J)=0$. This contradicts finiteness of $\lambda(\bfT)/\lambda(J)$. \end{rem}
\begin{rem}\label{Ghate} The quotient $\bfT/J$ measures $p$-adic congruences between Hecke eigenvalues of an eigenform corresponding to $\lambda_0$ and automorphic eigenforms corresponding to $\lambda \in \Pi$. It can be interpreted in terms of the congruence module $C(\bfT_0)$ as defined in \cite{Ghate02cong}, section 1. Indeed, the Hecke algebra $\bfT_0$ decomposes over $E$ as $\bfT_0 \otimes_{\Oo} E = X \oplus Y$ with $X=\lambda_{0}(\bfT_0)\otimes_{\Oo}E=E$ and $Y=\varphi_0(\bfT_0)\otimes_{\Oo}E=\bfT\otimes_{\Oo}E$. If we denote the corresponding projections as $\pi_X$ and $\pi_Y$ and set $\bfT^X:= \pi_X(\bfT_0)=\Oo$, $\bfT^Y=\pi_Y(\bfT_0)=\bfT$, $\bfT_X:= \bfT_0 \cap X$ and $\bfT_Y=\bfT_0\cap Y = 
J$ 
then the congruence module  $C(\bfT_0)$ is defined as $\frac{\bfT^X \oplus \bfT^Y}{\bfT}$. By Lemma 1 in [loc.cit.] we further know that $\bfT/J = \bfT^Y/\bfT_Y$ is isomorphic to $C(\bfT_0)$. 
\end{rem}
\begin{prop} \label{mainprop} 
For every $\lambda \in \Pi$ write $m_{\lambda}$ for the largest integer such that $\lambda_{0}(T) \equiv \lambda(T)$ mod $\varpi^{m_{\lambda}}$ for all $T \in \bfT_0$. 
Then \be \label{41} \frac{1}{e}\cdot \sum_{\lambda \in \Pi} m_{\lambda} \geq \val_p(\#\bfT/J).\ee
If $J$ is principal, then (\ref{41}) becomes an equality. \end{prop}
\begin{proof} By our assumption on $E$ the residue field condition in Theorem \ref{Kr12} is satisfied.
Recall that the structure map $\Oo \to \bfT/J$ 
is surjective. 
Hence as in the proof of Proposition \ref{Version 1} it follows that there exists a unique maximal ideal $\fm\subset \bfT$ 
containing $J$. 
 Write $\bfT_{\fm}$ for the localization of $\bfT$ at $\fm$.  Number the elements of $\Pi$ as $\lambda_1, \lambda_2, \dots, \lambda_r$. By renumbering the $\lambda_i$s we may assume that for $i=1,2,\dots, s\leq r$ the map $\lambda_{i}: \bfT \twoheadrightarrow \Oo$ factors through $\bfT_{\fm}$.  Then $m_{\lambda_i}=0$ for $s<i\leq r$. Now the Proposition follows from Theorem \ref{Kr12} (or Proposition \ref{principalJ} if $J$ is principal) upon taking $T=\bfT_{\fm}$, $T_i=\Oo$, $i=1,2,\dots, s$ 
 (i.e, by taking $n_1=n_2=\dots=n_s=1$ in section \ref{Commutative algebra}), $\varphi_i=\lambda_{i}: T \twoheadrightarrow T_i$. \end{proof}

\section{Further examples of applications to congruences between automorphic forms} \label{s5}

In this section we will present a few examples where our general result can be applied. We will use notation introduced in section \ref{Congruences - the general case}. 

\begin{example}[More general Eisenstein congruences]
Let $p$ be an odd prime and $\chi$ a Dirichlet character of order prime to $p$. Let the conductor of $\chi$ be $Np^r$ with $(N,p)=1$ (so $r=0$ or $1$). For each integer $k\geq 2$ such that $\chi(-1)=(-1)^k$ let $S_k(Np, \chi)$ be the space of cuspforms of weight $k$, level $Np$ and character $\chi$. 
Let $\mS_0$ be the complex vector space of modular forms $S_k(Np, \chi)  \oplus \bfC {\rm E}_{\chi}$, where $E_{\chi} \in M_k(Np,\chi)$ is the Eisenstein series with Hecke eigenvalues $1+\chi(\ell) \ell^{k-1}$ for $\ell \nmid Np$.

Let $\bfT_0$ be the $\Oo$-subalgebra of endomorphisms of $S_k(Np, \chi)  \oplus \bfC {\rm E}_{\chi}$  generated by the classical Hecke operators $T_{\ell}$ for $\ell \nmid Np$. 
Write $\lambda_0$ for the Hecke eigenvalue character corresponding to ${\rm E}_{\chi}$. Then $J$ is given by the Eisenstein ideal generated by $T_{\ell}-1-\chi(\ell)\ell^{k-1}$ for $\ell \nmid Np$ in the cuspidal Hecke algebra $\bfT$.

One expects that the order of the congruence module $\bfT/J$ is bounded from below by the order of the quotient $ \Oo/L(\chi,1-k)$, where $L(\chi,s)$ is the usual Dirichlet $L$-series attached to $\chi$ (case $k=2$ is proven in  Proposition 5.1 of \cite{SkinnerWiles97}, certain cases are implicit in e.g. \cite{MazurWiles84} and \cite{Wiles90}, general case is treated in an unpublished work of Skinner).
Assuming this bound we can apply Proposition \ref{mainprop} to conclude that $$\frac{1}{e} \cdot \sum_{\lambda \in \Pi} m_{\lambda} \geq {\rm val}_p(\#\Oo/L(\chi,1-k)).$$

\begin{rem}

A similar result on higher Eisenstein congruences can also be proven for modular forms over imaginary quadratic fields. In this case one also does not 
generally have principality of $J$, so Proposition \ref{principalJ} cannot be used and one only gets the inequality from Theorem \ref{Kr12}. 
For details we refer the reader to the upcoming work of the first two authors \cite{BergerKlosin13preprint}.
\end{rem}

\end{example}

\begin{example}[Congruence primes for  primitive forms]
Let $k \geq 2, N \geq 1$ be integers and write $\mS$ for the $\bfC$-space of elliptic modular forms of weight $2$ and level $N$. Let $p \geq 5$ be a prime. Let $f\in \mS$ be a \emph{primitive} form, i.e. an eigenform for the Hecke operators and a newform.
 Let $f_1, \ldots, f_{r'}$ be eigenforms spanning the subspace of $\mS$ orthogonal to $f$ under the Petersson inner product. Let $f_1, \dots, f_r$ be a maximal subset of the above (after possibly renumbering the $f_i$s) with the property that no pair of the $f_i$s shares the same eigenvalues for all Hecke operators away from primes $\ell \mid N$. As in Section \ref{s3} we write $\varpi_N$ for a uniformizer of the ring of integers of the composition of all the coefficient fields, and $e_N$ for the ramification degree of $\Oo_N$ over $\bfZ_p$. Let $\varpi^{m_i}_N$ be the highest power of $\varpi_N$ such that the Hecke eigenvalues of $f$ are congruent to $f_i$ modulo $\varpi^{m_i}_N$ for Hecke operators $T_{\ell}$ for all primes $\ell\nmid N$. 
Let $\bfT_0$ (resp. $\bfT_0^N$) be the $\Oo$-subalgebra acting on the space $\mS_0=\bigoplus_{i=1}^r \bfC f_i \oplus \bfC f$ generated by the Hecke operators $T_{\ell}$ for $\ell \nmid N$ (resp. for all $\ell$).

Assume now  that $f$ is ordinary at $p$ and that the $p$-adic Galois representation associated to $f$ is residually absolutely irreducible when restricted to the absolute Galois group of  $\bfQ(\sqrt{(-1)^{(p-1)/2}p})$. 
Under some mild technical assumptions 
Hida (see e.g. Theorem 5.20 and (Cg2) in \cite{Hida00}) proved that the $p$-valuation of the order of the congruence module $C(\bfT_0^N)$ (as defined in Remark \ref{Ghate}) equals the $p$-valuation of $\#\Oo/L^{\rm alg}(1,{\rm Ad}(f))$, where $L^{\rm alg}(1,{\rm Ad}(f))$ denotes the algebraic part of  the value at $s=1$ of the adjoint $L$-function attached to $f$ (for details see \cite{Hida00}).

By Proposition \ref{mainprop} we can conclude that  $$\frac{1}{e_N} \cdot \sum_{i=1}^r m_i \geq {\rm val}_p(\# C(\bfT_0)).$$
If including the operators $T_{\ell}$ for primes $\ell\mid N$ does not affect the depth of the congruences, i.e., if $\#C(\bfT_0)=\#C(\bfT_0^N)$ we also get
$$\frac{1}{e_N} \cdot \sum_{i=1}^r m_i \geq {\rm val}_p(\# (\Oo/L^{\rm alg}(1,{\rm Ad}(f)))).$$


\end{example}

\begin{example}[Congruences to Saito-Kurokawa lifts]  Let $k$ be an even positive integer and let $p>k$ as before be a fixed prime. 
Let $f \in S_{2k-2}(\SL_2(\bfZ))$ be a newform and write $F_f$ for the Saito-Kurokawa lift of $f$. Write $\mS$ for the $\bfC$-space of Siegel cusp forms of weight $k$ and full level. Then $F_f \in \mS$. Write $\mS^{\rm SK} \subset \mS$ for the subspace spanned by Saito-Kurokawa lifts.  
Let $\mS_0=\mS_{\rm SK} \oplus \bfC F_f$, where $\mS_{\rm SK}$ is the orthogonal complement of $\mS^{\rm SK}$ with respect to the standard Petersson inner product on $\mS$.  
Set $\Sigma=\{p\}$.  
Let $\bfT_0$ be the ($\Oo$-base change of the) quotient of the standard Siegel Hecke algebra $\mR_{\Sigma}$ (away from $p$) acting on $\mS_0$. Let $\lambda_0$ be the Hecke eigencharacter corresponding to $F_f$. Then $\bfT$ can be identified with the quotient of $\mR_{\Sigma}$ acting on $\mS_{\rm SK}$ and $\bfT/J$ measures congruences between the Hecke eigenvalues corresponding to $F_f$ and those corresponding to Siegel eigenforms which are not Saito-Kurokawa lifts (cf. \cite{Brown07} and \cite{Brown11} for details).

From now on assume that $f$ is ordinary at $p$ and that the $p$-adic Galois representation attached to $f$ is residually absolutely irreducible. It has been shown by Brown (\cite{Brown11}, Corollary 5.6) that under some mild assumptions the order of the finite quotient 
$\bfT/J$ is bounded from below by the order of the quotient $ \Oo/L^{\rm alg}(k,f)$, where $L^{\rm alg}(k,f)$ denotes the algebraic part of the special value at $k$ of the standard $L$-function of $f$ (for details see [loc.cit.]). Hence we can apply Proposition \ref{mainprop} to conclude that \be \label{SK} \frac{1}{e} \cdot \sum_{\lambda \in \Pi} m_{\lambda} \geq {\rm val}_p(\#\Oo/L^{\rm alg}(k,f)).\ee Since strong multiplicity one holds for forms in $\mS$, the elements of $\Pi$ are in one-to-one correspondence with eigenforms (up to scalar multiples) in $\mS_{\rm SK}$. So, (\ref{SK}) tells us that the ``depths'' $m$ of the (mod $\varpi^m$) congruences to $F_f$ of the eigenforms in $\mS_{\rm SK}$ add up to no less than the $\varpi$-adic valuation of the standard $L$-value of $f$ at $k$. 
\end{example}

\begin{rem} There are several other situations where similar conclusions can be drawn (for example in the context of congruences among automorphic forms on the unitary group $\U(2,2)$ and the CAP ideal \cite{Klosin09, Klosin12preprint} or the Yoshida congruences on $\Sp_4$ \cite{AgarwalKlosin12preprint}). Since the reasoning is verbatim to the examples listed above we leave their formulation to the interested reader. \end{rem}

\bibliographystyle{amsalpha}
\bibliography{standard2}

\end{document}